\documentclass[12pt]{article} 
\usepackage{amsfonts, amsmath, amsthm, amssymb, bbm, dsfont, enumitem, hyperref, xcolor}
\usepackage[margin=1in]{geometry}

\usepackage{graphicx}
\usepackage{url}
\hypersetup{
  colorlinks   = true, %Colours links instead of ugly boxes
  urlcolor     = blue, %Colour for external hyperlinks
  linkcolor    = blue, %Colour of internal links
  citecolor   = red %Colour of citations
}

\newtheorem{theorem}{Theorem}
\newtheorem{lemma}[theorem]{Lemma}
\newtheorem{claim}[theorem]{Claim}
\newtheorem{proposition}[theorem]{Proposition}
\newtheorem{corollary}[theorem]{Corollary}

\theoremstyle{definition}

\theoremstyle{remark}

\newcommand{\bea}{\begin{eqnarray}}
\newcommand{\eea}{\end{eqnarray}}
\newcommand{\<}{\langle}
\renewcommand{\>}{\rangle}

\def\eps{{\varepsilon}}

\def\cF{{\mathcal F}}

\def\de{{\rm d}}

\def\<{\langle}
\def\>{\rangle}

\def\cM{{\cal M}}

\def\b0{{\boldsymbol{0}}}

\def\cS{{\mathcal S}}

\renewcommand{\b}{\mathbf{b}}

\def\fr{\frac}
\def\lt{\left}
\def\rt{\right}

\def\eps{\varepsilon}

\def\bbE{{\mathbb{E}}}

\def\bbR{{\mathbb{R}}}

\def\cF{{\mathcal{F}}}

\def\cP{{\mathcal{P}}}

\DeclareMathOperator*{\EE}{\bbE}

\title{Improved Lower Bound for Frankl's Union-Closed Sets Conjecture}
\author{Ryan Alweiss\thanks{Department of Mathematics and Mathematical Statistics, University of Cambridge.} \and 
Brice Huang\thanks{Department of Electrical Engineering and Computer Science, Massachusetts Institute of Technology.} 
\and 
Mark Sellke\thanks{Department of Statistics, Harvard University.}
}

\date{}

\begin{document}

\maketitle

\begin{abstract}
    We verify an explicit inequality conjectured in \cite{gilmer2022constant}, thus proving that for any nonempty union-closed family $\cF \subseteq 2^{[n]}$, some $i\in [n]$ is contained in at least a $\fr{3-\sqrt{5}}{2} \approx 0.38$ fraction of the sets in $\cF$. One case, an explicit one-variable inequality, is checked by computer calculation.
\end{abstract}
\section{Introduction}

Let $\cM_\phi$ be the set of probability measures $\mu \in \cP([0,1])$ with expectation $\phi$. 
Define
\begin{equation}
\label{eq:F}
    F(\mu) = \EE_{(x,y)\sim \mu \times \mu} H(xy) - \EE_{x\sim \mu} H(x)
\end{equation}
where $H(x) = -x\log x - (1-x)\log (1-x)$ is the entropy function and $\log$ denotes the natural logarithm.
Note that $F$ is continuous in the weak topology and $\cM_\phi$ is compact, so $F$ has a minimizer over $\cM_\phi$.
In this note, we will show the following results.
\begin{theorem}
    \label{thm:main}
    For all $\phi\in [0,1]$, the minimum of $F(\mu)$ over $\cM_\phi$ is attained at some $\mu$ supported on at most two points.
    Furthermore, if a minimizer is supported on exactly two points, then one of the points is $0$.
\end{theorem}
The case of $\mu$ supported on $\{0,x\}$ leads to the following definition:
\[
    S = \lt\{
        \phi \in [0,1] : \phi H(x^2) \ge xH(x) ~\forall x\in [\phi,1]
    \rt\},
    \quad 
    \phi^\ast = \min(S).
\]
Note that the condition defining $S$ is monotone in $\phi$ and $S$ is clearly closed, so $\min(S)$ is well defined. As in the recent breakthrough \cite{gilmer2022constant} by Gilmer, a bound on Frankl's union-closed conjecture follows from the above. 

\begin{theorem}
    \label{thm:union-closed}
    The union-closed conjecture holds with constant $1-\phi^\ast$, i.e. for any non-empty union-closed family $\cF\subseteq 2^{[n]}$, some $i\in [n]$ is contained in at least $1-\phi^\ast$ fraction of the sets in $\cF$.
\end{theorem}

Throughout this paper we set $\varphi = \fr{\sqrt5-1}{2}$.
In the Appendix, we give a numerical verification of the following claim.
We require certain computer calculations (detailed in an attached Python file) to be accurate to within margin of error $10^{-3}$, which can be made completely rigorous using interval arithmetic.
\begin{claim}
    \label{claim:ineq}
    If $x \in [\varphi,1]$, then $\varphi H(x^2) \ge xH(x)$, with equality if and only if $x\in \{\varphi,1\}$.
\end{claim}
Assuming Claim~\ref{claim:ineq}, the following claim identifies the value of $\phi^\ast$.
Then, Theorem~\ref{thm:union-closed} implies that the union-closed conjecture holds with constant $1-\varphi = \fr{3-\sqrt5}{2}$. This is a natural barrier for the method of \cite{gilmer2022constant} as explained therein. Interestingly, Claim~\ref{claim:ineq} has been mentioned previously in a different context by \cite{boppana1985amplification}.
\begin{claim}
    \label{claim:numeric}
    We have that $\phi^\ast = \varphi$. 
\end{claim}

\paragraph{Related Work.} 
The union-closed conjecture has been the subject of much study, see \cite{polymath,knill1994graph,wojcik1999union,balla2013union,karpas2017two} or the survey \cite{bruhn2015journey}.
The recent breakthrough \cite{gilmer2022constant} by Gilmer showed that this conjecture holds with constant $0.01$.

Concurrently with and independently of this work, Chase and Lovett \cite{chase2022approximate}, Sawin \cite{sawin2022improved}, and Pebody \cite{pebody2022extension} also proved the union-closed conjecture with constant $\fr{3-\sqrt5}{2}$.
\cite{sawin2022improved} also outlined an argument to improve this bound by an additional small constant, which was subsequently made explicit in \cite{yu2023dimension} (using Lemma~\ref{lem:concave} below) and \cite{cambie2022better}.
Moreover, \cite{sawin2022improved} and Ellis \cite{ellis2022note} found counterexamples to \cite[Conjecture 1]{gilmer2022constant}, which would have implied the full union-closed conjecture with constant $\fr12$.

\paragraph{Acknowledgements.}
We thank Zachary Chase and Shachar Lovett for sharing their writeup \cite{chase2022approximate} with us. We thank Mehtaab Sawhney and the anonymous referee for helpful comments. RA was supported by an NSF Mathematical Sciences Postdoctoral Research Fellowship. BH was supported by an NSF Graduate Research Fellowship, a Siebel scholarship, NSF awards CCF-1940205 and DMS-1940092, and NSF-Simons collaboration grant DMS-2031883.

\section{Reduction to Two Point Masses}

\begin{lemma}
    \label{lem:concave}
    $F$ is concave on $\cM_\phi$ for any $\phi\in [0,1]$, i.e.
    \begin{equation}
    \label{eq:pF}
        pF(\mu_1)+(1-p) F(\mu_2)
        \leq
        F(p\mu_1+(1-p)\mu_2)\quad\forall~\mu_1,\mu_2\in\cM_{\phi},~p\in [0,1]\,.
    \end{equation}
\end{lemma}
\begin{proof}
    Let $\gamma(x) = \mu([0,x])$ be the cumulative distribution function of $\mu$. Thus $\gamma(1)=1$ and
    \[
    	\phi = \int_0^1 x\mu(\de x) = 1-\int_0^1 \gamma(x)~\de x\,,
    \]
    so 
    \begin{equation}
        \label{eq:phi}
        \int_0^1 \gamma(x)~\de x = 1-\phi\,.
    \end{equation}
    Using integration by parts,
    \[
    	\int_0^1 H(x)\mu(\de x)
    	=
    	H(x)\gamma(x)\bigg|_0^1 - \int_0^1 H'(x)\gamma(x)~\de x
    	=
    	\int_0^1 \lt(\log \fr{x}{1-x}\rt) \gamma(x)~\de x\,.
    \]
    Similarly,
    \begin{align*}
    	\int_0^1 H(xy) \mu(\de y)
    	&=
    	H(xy) \gamma(y) \bigg|_0^1 - \int_0^1 xH'(xy) \gamma(y)~\de y \\
    	&= H(x) + \int_0^1 \lt(x \log \fr{xy}{1-xy}\rt) \gamma(y)~\de y\,; \\
    	\int_0^1 \lt(x \log \fr{xy}{1-xy}\rt)\mu(\de x)
    	&=
    	\lt(x \log \fr{xy}{1-xy}\rt) \gamma(x) \bigg|_0^1 - \int_0^1 \fr{\de}{\de x} \lt(x \log \fr{xy}{1-xy}\rt) \gamma(x)~\de x\,, \\
    	&= \log \fr{y}{1-y} - \int_0^1 \lt(\fr{1}{1-xy} + \log \fr{xy}{1-xy}\rt) \gamma(x)~\de x\,; \\
    	\iint_{[0,1]^2} H(xy) \mu(\de x)\mu(\de y)
    	&=
    	\int_0^1 H(x) \mu(\de x) + \int_0^1 \gamma(y) \int_0^1 x \log \fr{xy}{1-xy}\mu(\de x)~\de y\,, \\
    	&=
    	2\int_0^1 \lt(\log \fr{x}{1-x}\rt) \gamma(x)~\de x - \iint\limits_{[0,1]^2} \lt(\fr{1}{1-xy} + \log \fr{xy}{1-xy}\rt) \gamma(x)\gamma(y) \de x\de y.
    \end{align*}
    So, letting $F(\gamma) = F(\mu)$ by slight abuse of notation, we have
    \[
    	F(\gamma)
    	=
    	\int_0^1 \lt(\log \fr{x}{1-x}\rt) \gamma(x)~\de x
    	- \iint_{[0,1]^2} \lt(\log x + \log y + \fr{1}{1-xy} + \log \fr{1}{1-xy}\rt) \gamma(x)\gamma(y) ~\de x~\de y\,.
    \]
    We will show this is concave in $\gamma$.
    The first integral is manifestly linear in $\gamma$, and the contributions of $\log x$ and $\log y$ are linear because, in light of \eqref{eq:phi}, 
    \[
        \iint_{[0,1]^2} (\log x) \gamma(x)\gamma(y) ~\de x~\de y
        = 
        (1-\phi) \int_0^1 (\log x) \gamma(x) ~\de x\,.
    \]
     After removing these terms, we are reduced to showing convexity of
    \[
        \iint_{[0,1]^2}
        \lt(\frac{1}{1-xy}
        +
        \log\frac{1}{1-xy}\rt)
        \gamma(x) \gamma(y)
        \de x\de y\,.
    \]
    Note that both $\frac{1}{1-xy}$ and $\log\frac{1}{1-xy}$ are of the form $\sum_{k\geq 0} a_k x^k y^k$ for constants $a_k\geq 0$. Hence it suffices to prove convexity of
    \[
        \iint_{[0,1]^2} x^k y^k \gamma(x) \gamma(y) \de x\de y
        =
        \lt(\int_0^1 x^k \gamma(x)\de x\rt)^2
    \]
    for any $k\geq 0$. This is the square of a linear function of $\gamma$, and hence is convex. (Note that all integrands are in $L^1$ and so there are no convergence issues.)
\end{proof}

\begin{lemma}
    \label{lem:at-most-two}
    $\arg\min_{\mu\in\cM_{\phi}}F(\mu)$ contains some $\mu$ supported on at most two points.
\end{lemma}
\begin{proof}
This follows immediately from Lemma~\ref{lem:concave} and the Krein-Milman theorem since $\cM_{\phi}$ is compact in the weak-* topology and convex, and all extreme measures in $\cM_{\phi}$ are supported on $1$ or $2$ points (see e.g. \cite{winkler1988extreme} for more on the latter point).  

For self-containedness, we also include an explicit and elementary version of this argument. First let $\mu\in\cM_{\phi}$ be any minimizer of $F$ and note that $\mu$ can be approximated arbitrarily well in the weak topology by $\hat\mu$ with finite support. In particular for any $\eps>0$, there exists $\hat\mu\in\cM_{\phi}$ with $F(\hat\mu)\leq F(\mu)+\eps$ of the form
\[
    \hat\mu(a_i)=b_i-b_{i-1},\quad 1\leq i\leq k
\]
for constants
$0\le a_1 < \cdots < a_k\leq 1$ and $0=b_0<b_1<\cdots<b_k=1$. We claim that for any $\eps>0$, the minimal $k$ such that such a $\hat\mu$ exists is at most two. Indeed given such a $\hat\mu$ with $k\geq 3$, we may consider $\hat\mu_{\eta}$ defined by 
\begin{align*}
    \hat\mu_{\eta}(a_1)&=b_1-b_0+\eta(a_3-a_2),
    \\
    \hat\mu_{\eta}(a_2)&=b_2-b_1-\eta(a_3-a_1),
    \\
    \hat\mu_{\eta}(a_3)&=b_3-b_2+\eta(a_2-a_1),
    \\
    \hat\mu_{\eta}(a_i)&=\hat\mu(a_i)=b_i-b_{i-1},\quad\forall \, i\in \{4,5,\dots,k\}.
\end{align*}
Then there exist $c_1,c_2>0$ such that $\hat\mu_{\eta}\in\cM_{\phi}$ if and only if $-c_1\leq \eta\leq c_2$; moreover the map $\eta\mapsto F(\hat\mu_{\eta})$ is concave by Lemma~\ref{lem:concave}. It is easy to see that both $\hat\mu_{-c_1},\hat\mu_{c_2}$ have support size at most $k-1$, and at least one of $F(\hat\mu_{-c_1}),F(\hat\mu_{c_2})$ is at most $F(\hat\mu)$ by concavity. Iterating this argument, we find a $\tilde\mu\in\cM_{\phi}$ with support size at most $2$ and with $F(\tilde\mu)\leq F(\hat\mu)\leq F(\mu)+\eps$. Taking a subsequential weak limit of the resulting $\tilde\mu$ as $\eps\to 0$ completes the proof.
\end{proof}

\section{Optimization over Two Point Masses}

\begin{lemma}
    \label{lem:not-two}
    If $\mu$ is supported on exactly two points, neither of which is $0$, then $\mu$ is not a minimizer of $F$ over $\cM_\phi$. 
\end{lemma}
\begin{proof}
    Suppose $\mu=p\delta_x+(1-p)\delta_y$ is a minimizer for $F$ over $\cM_{\phi}$ for $0< y<x< 1$ distinct and $0<p<1$. Then any $z\in [0,1]$ can be written as $z=qx+(1-q)y$ for some $q\in\bbR$ (which may be negative). We have 
    \[
        \mu+t\delta_z-tq\delta_x-t(1-q)\delta_y\in \cM_{\phi}
    \]
    for sufficiently small $t\geq 0$ and so
    \[
        \lim_{t\to 0^+} \frac{F\big(\mu+t\delta_z-tq\delta_x-t(1-q)\delta_y\big)-F(\mu)}{t}
        \geq 0.
    \]
    It is not difficult to see from the definition \eqref{eq:F} of $F$ that the left-hand limit equals
    \begin{equation}
        \label{eq:f}
        f(z)-qf(x)-(1-q)f(y)\geq 0\,,
    \end{equation}
    for
    \[
        f(w):=
        2[pH(xw)+(1-p)H(yw)]-H(w).
    \]
    Equation \eqref{eq:f} implies that $f$ lies above the line passing through $(x, f(x))$ and $(y, f(y))$. 
    Since $f$ is a smooth function and $x,y$ are in the interior of $[0,1]$, we deduce that
    \begin{enumerate}[label=(\alph*)]
        \item \label{itm:f'} $f'(x)=f'(y)= \frac{f(x)-f(y)}{x-y}$, and
        \item \label{itm:f''} $f''(x),f''(y)\ge 0$. 
    \end{enumerate}
    Moreover, \ref{itm:f'} implies
    \begin{enumerate}[label=(\alph*), resume]
        \item $f''(z)\le 0$ for some $z\in [y,x]$.
    \end{enumerate}
    However we compute using $H'(w)=\log\frac{1-w}{w}$ that:
    \begin{align*}
        f'(z)
        &=
        2\lt[px\log\frac{1-xz}{xz}+(1-p)y\log\frac{1-yz}{yz}\rt]-\log\frac{1-z}{z}\,,
        \\
        f''(z)
        &=
        -2\lt[\frac{px}{z(1-xz)}+\frac{(1-p)y}{z(1-yz)}\rt]+\frac{1}{z(1-z)}\,.
    \end{align*} 
    Note that $g(z):= z(1-z)(1-xz)(1-yz)f''(z)$ has the same sign as $f''(z)$ and is a quadratic function in $z$ with leading coefficient 
    \[
        -2pxy-2(1-p)xy+xy=-xy
        % \leq 
        <0\,.
    \]
    Hence the inequalities $g(x),g(y)\geq 0$ and $g(z)\leq 0$ can hold only if $g$ and hence $f''$ vanishes on the entire interval $[x,y]$. 
    This is impossible since we just saw $g$ has non-zero leading coefficient.

    The case $x=1$, $y>0$ is very similar. 
    While we have $f''(y)\ge 0$ as above, since $1$ is not in the interior of $[0,1]$, we cannot immediately deduce that $f''(1)\ge 0$. 
    However in this case $g(z)$ is a multiple of $1-z$, and so $g(1)=0\geq 0$.
    Then the same argument applies: $g(z)$ is a quadratic polynomial with negative leading coefficient $-y<0$. 
    Because $g$ takes non-negative values at $y$ and $1$, it takes positive values in between. 
    However since $f$ is continuous on $[0,1]$ and smooth on $(0,1)$, and stays above the line segment through $(y,f(y))$ and $(1,f(1))$, it must have non-positive second derivative at some $z\in (y,1)$. 
    Since $g$ and $f''$ have the same sign on $(0,1)$, this is a contradiction. 
    (Note that $f''(1)$ does not actually exist if $x=1$ and is not used in this argument.)
\end{proof}

\section{Conclusion}

\begin{proof}[Proof of Theorem~\ref{thm:main}]
    Follows from Lemmas~\ref{lem:at-most-two} and \ref{lem:not-two}. 
\end{proof}

\begin{lemma}
    \label{lem:phi*-ge-varphi}
    We have that $\phi^\ast \ge \varphi$. 
\end{lemma}
\begin{proof}
    Note that $H(\varphi^2) = H(\varphi)$. 
    If $\phi < \varphi$, then $\phi H(\varphi^2) < \varphi H(\varphi)$, and so $\phi \not\in S$. 
\end{proof}

\begin{corollary}
    \label{cor:main-larger-than-phi*}
    If $\phi \ge \phi^\ast$, then $F(\mu) \ge 0$ for all $\mu \in \cM_\phi$.
\end{corollary}
\begin{proof}
    By Theorem~\ref{thm:main}, it suffices to check $F(\mu) \ge 0$ for $\mu = p\delta_x + (1-p)\delta_0$ with $p = \phi/x$ and $x\in [\phi,1]$ (this includes the case $\mu = \delta_\phi$, corresponding to $x=\phi$).
    By monotonicity of the condition defining $S$, $\phi \in S$.
    So, 
    \[
        F(\mu) = \fr{\phi^2}{x^2}H(x^2) - \fr{\phi}{x} H(x) = \fr{\phi}{x^2} (\phi H(x^2) - xH(x)) \ge 0\,.
    \qedhere
    \]
\end{proof}

From Theorem~\ref{thm:main}, we deduce the following tight version of Gilmer's \cite[Lemma 1]{gilmer2022constant}.
Theorem~\ref{thm:union-closed} follows from Corollary~\ref{cor:gilmer-tight} by the same argument as in \cite[Proof of Theorem 1]{gilmer2022constant}.
We recall Gilmer's ingenious insight was that given a union-closed family $\cF \subseteq 2^{[n]}$, if $A,A'$ are independent uniformly random samples from $\cF$, then $A\cup A'\in\cF$ is \emph{not} uniformly random and thus has strictly smaller entropy. On the other hand, Corollary~\ref{cor:gilmer-tight} can be applied element-by-element to show that $A\cup A'$ actually has equal or larger entropy.

\begin{corollary}
    \label{cor:gilmer-tight}
    Suppose $\{p_c\}_{c\in \cS} \subset [0,1]$ is a finite sequence of real numbers and $c$ is a random variable supported on $\cS$ such that $\EE_c[p_c] \le 1-\phi^\ast$.
    If $c'$ is an independent copy of $c$, then
    \[
        \EE_{c,c'}[H(p_c+p_{c'}-p_cp_{c'})] \ge \EE_c [H(p_c)]\,.
    \]
\end{corollary}
\begin{proof}
   Let $\mu$ be the distribution of $x=1-p_c$.
   Let $\phi = \EE_{x\sim \mu} [x]$, so $\phi > \phi^\ast$. 
   By Corollary~\ref{cor:main-larger-than-phi*},
   \[
        \EE_{c,c'}[H(p_c+p_{c'}-p_cp_{c'})] - \EE_c [H(p_c)]
        =
        \EE_{(x,y)\sim \mu\times \mu}H(xy) - \EE_{x\sim \mu} H(x)
        = F(\mu)
        \ge 0\,.
    \qedhere
   \]
\end{proof}

Finally, we verify Claim~\ref{claim:numeric} assuming Claim~\ref{claim:ineq}.
\begin{proof}[Proof of Claim~\ref{claim:numeric}]
    Claim~\ref{claim:ineq} implies $\varphi \in S$, so $\phi^\ast \le \varphi$ by definition of $\phi^\ast$. 
    On the other hand, Lemma~\ref{lem:phi*-ge-varphi} gives $\phi^\ast \ge \varphi$.
\end{proof}

\bibliographystyle{acm}
\bibliography{bib}

\appendix

\section{Proof of Claim~\ref{claim:ineq}}

In this appendix, we prove Claim~\ref{claim:ineq}.
Throughout this appendix, we use Claims to indicate results requiring the correctness of computer outputs within margin of error $10^{-3}$ or greater. The only computations which rely on a computer are the entries in Tables~\ref{tab:Lbash} and \ref{tab:Gbash}.
Figure~\ref{fig:plot} plots the function
\[
    G(x) = \varphi H(x^2) - xH(x),
\]
from which Claim~\ref{claim:ineq} can be checked visually.
\begin{figure}
    \centering
    \includegraphics[width=0.5\textwidth]{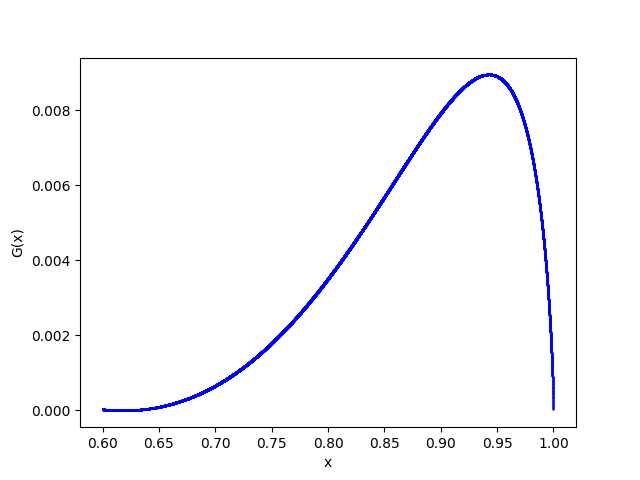}
    \caption{Plot of $G(x)$ for $x\in [0.6,1]$. Claim~\ref{claim:ineq} states the minimum value of $0$ on $x\in [\varphi,1]$ is achieved precisely at the endpoints $x\in \{\varphi,1\}$.}
    \label{fig:plot}
\end{figure}
We show below that, assuming correctness of certain computer calculations to within margin of error $10^{-3}$,
\[
    G(x)\geq 0,\quad \forall x\in [\varphi,1].
\]
We verify this separately on the intervals $I_1=[\varphi,0.77],I_2=[0.76,0.98],I_3=[0.98,1]$.

\subsection{Verification on $I_1$}

We first compute the derivative of $G$:
\begin{align*}
    G'(x)=&2x \varphi \log \fr{1-x^2}{x^2} - H(x) - x \log \fr{1-x}{x} \\
    &= 2x \varphi \log \fr{1-x^2}{x^2} + x \log x + (1-x) \log (1-x) + x \log x - x \log (1-x) \\
    &= 2x \varphi \log \fr{1-x^2}{x^2} + 2x \log x + (1-2x) \log (1-x) \\
\end{align*}
Note that $G(\varphi)=G'(\varphi)=0$, the latter since
\begin{align*}
    G'(\varphi)
    &=
    2\varphi^2 \log(1/\varphi)+2\varphi \log(\varphi)+(1-2\varphi)\log(\varphi^2)
    \\
    &=
    (-2\varphi^2 +2\varphi+2(1-2\varphi))\log\varphi
    \\
    &=
    2(1-\varphi-\varphi^2)\log(\varphi)=0.
\end{align*}
\begin{claim}
\label{claim:I1}
    Claim~\ref{claim:ineq} holds on $I_1=[\varphi,0.77]$. 
\end{claim}
\begin{proof}
    As $G(\varphi)=G'(\varphi)=0$, it suffices to verify that $G$ is convex on $I_1$.
    It is not hard to check that its second derivative equals $G''(x) = L(x)/(1-x^2)$, where
    \[
        L(x)
        :=
        2\varphi(1-x^2)\log(x^{-2}-1)-4\varphi-2x^2 \log x + 2(x^2-1)\log(1-x)+x+2\log(x)+1\,.
    \]
    We now estimate the Lipschitz constant of each non-constant term of $L$ on $x\in I_1$. For the first term,
    \begin{equation}
    \label{eq:L-first-term}
    \begin{aligned}
        \lt|\frac{d}{dx}
        \big(2\varphi(1-x^2)\log(x^{-2}-1)\big)\rt|
        &\leq
        2\varphi\sup_{x\in I_1}
        \big(
        |2x^3|+2|x\log(x^{-2}-1)|
        \big)
        \\
        &\leq
        2\varphi(1.1+1.6\cdot \log(2))
        \\
        &\leq
        2\varphi\cdot 2.3
        \leq 
        3
    \end{aligned}
    \end{equation}
    since $\log(2)\leq 0.75$ and $\varphi\leq 5/8$. Next,
    \begin{align*}
        \lt|\frac{d}{dx}
        \big(2x^2\log(x)\big)\rt|
        &\leq
        \sup_{x\in I_1}|4x\log(x)+2x|
        \\
        &\leq
        1.6\sup_{x\in I_1}|2\log(x)+1|
        \\
        &\leq
        1.6
    \end{align*}
    since $\log(x)\in [-1,0]$ for all $x\in 
    I_1$. Continuing, using $\log(5)\leq 2$,
    \begin{align*}
        \lt|\frac{d}{dx}
        \big(2(x^2-1)\log(1-x)\big)\rt|
        &\leq
        2\sup_{x\in I_1}|2x\log(1-x)-\frac{x^2-1}{1-x}|
        \\
        &\leq
        2\sup_{x\in I_1}|2x\log(1-x)+x+1|
        \\
        &\leq
        2\cdot\max(1.6\log(5),1.8)
        \\
        &\leq
        2\cdot 1.6\cdot 2
        =
        6.4.
    \end{align*}
    Finally $\frac{d}{dx}(x)=1$ and $\frac{d}{dx}(2\log x)=2/x\leq 3.5$. 
    Combining, we find that $L(x)$ restricted to $I_1$ has Lipschitz constant at most
    \[
        % 1.6+6.4+1+\max(3,3.5)\leq 12.5.
        1.6+6.4+1+3+3.5\leq 15.5.
    \]
    Therefore to show $G$ is convex and hence non-negative on $I_1=[\varphi,0.77]$ it suffices to exhibit a $\frac{1}{400}$-dense subset of $I_1$ on which $L(x)=(1-x^2)G''(x)\geq 0.04\geq \frac{15.5}{400}$. In Table~\ref{tab:Lbash} below we compute the values of $L$ on each multiple of $\frac{1}{200}$ from $0.6$ to $0.77$ inclusive. We find that $L(x)\geq 0.09$ holds at all of these points, completing the numerical verification on $I_1$.
\end{proof}
\begin{table}[h!]
    \centering
    \footnotesize
    \begin{tabular}{|c|c||c|c||c|c||c|c||c|c||c|c|} 
        \hline
        $x$ & $L(x)$ & $x$ & $L(x)$ & $x$ & $L(x)$ & $x$ & $L(x)$ & $x$ & $L(x)$ & $x$ & $L(x)$ \\ \hline \hline
        0.600 & 0.1020 & 0.630 & 0.1117 & 0.660 & 0.1173 & 0.690 & 0.1182 & 0.720 & 0.1137 & 0.750 & 0.1032 \\ \hline
        0.605 & 0.1039 & 0.635 & 0.1130 & 0.665 & 0.1178 & 0.695 & 0.1178 & 0.725 & 0.1124 & 0.755 & 0.1009 \\ \hline
        0.610 & 0.1057 & 0.640 & 0.1141 & 0.670 & 0.1182 & 0.700 & 0.1173 & 0.730 & 0.1109 & 0.760 & 0.0983 \\ \hline
        0.615 & 0.1074 & 0.645 & 0.1151 & 0.675 & 0.1184 & 0.705 & 0.1167 & 0.735 & 0.1093 & 0.765 & 0.0955 \\ \hline
        0.620 & 0.1089 & 0.650 & 0.1159 & 0.680 & 0.1185 & 0.710 & 0.1159 & 0.740 & 0.1075 & 0.770 & 0.0925 \\ \hline
        0.625 & 0.1104 & 0.655 & 0.1167 & 0.685 & 0.1184 & 0.715 & 0.1149 & 0.745 & 0.1054 &        &        \\ \hline
    \end{tabular}
    \caption{Evaluations of $L$ to precision $10^{-4}$. All values appear to be at least $0.09$, and it suffices for all values to be at least $0.04$.}
    \label{tab:Lbash}
\end{table}

\subsection{Verification on $I_2$}

Our verification for $x\in I_2$ is based on evaluating $G$. 
We write $G(x)=g_1(x)-g_2(x)$ for
 \begin{align*}
    g_1(x) &= \varphi H(x^2),
    \\
    g_2(x) &= xH(x).
\end{align*}
Note that $g_1$ is clearly decreasing on $I_2$. The next lemma shows the same for $g_2$.
\begin{lemma}
    \label{lem:f2-decr}
    $g_2$ is decreasing on $[5/7,1]\supseteq I_2$.
\end{lemma}
\begin{proof}
    First we claim that it suffices to show $g_2'(5/7)\leq 0$. This is because
    \begin{align*}
    g_2'(x)=H(x)+x\log\frac{1-x}{x}
    =
    2x\log\frac{1}{x}-(2x-1)\log\frac{1}{1-x}
    \end{align*}
    so $g_2'(x)\leq 0$ if and only if 
    \begin{equation}
    \label{eq:f2'-condition}
    \lt(1-\frac{1}{2x}\rt) \log\frac{1}{1-x}\geq \log\frac{1}{x}.
    \end{equation}
    Indeed both terms on the left-hand side are increasing while the right-hand side is decreasing.

    It remains to show that $g_2'(5/7)\leq 0$ which in light of \eqref{eq:f2'-condition} is equivalent to showing
    \[
    \frac{3}{10} \log(7/2)\geq \log(7/5),
    \]
    i.e. $(7/5)^{10/3}\leq 7/2$. This holds since $(7/5)^3 \leq 2(7/5)=14/5$ and $7/5\leq \lt(\frac{5}{4}\rt)^3 =\lt(\frac{7/2}{14/5}\rt)^3$.
\end{proof}

\begin{claim}
\label{claim:I2}
    Claim~\ref{claim:ineq} holds for $x\in I_2$. 
\end{claim}
\begin{proof}
    We computer-evaluate $g_1,g_2$ at a finite set of values $x_1<x_2<\dots<x_{97}$ with $5/7<x_1<0.76$ and $x_{97}=0.98$ and verify that $g_1(x_{i+1})\geq g_2(x_i)$ for each $i$. 
    The values are shown in Table~\ref{tab:Gbash}; note that in all cases $g_1(x_{i+1})- g_2(x_i)\ge \frac{2}{1000}$ holds, modulo rounding to four decimal places.
    The intervals $[x_i,x_{i+1}]$ cover $I_2$, and for all $x\in [x_i,x_{i+1}]$ we have
    \[
        g_2(x) \le g_2(x_i) \le g_1(x_{i+1}) \le g_1(x)\,.
    \qedhere
    \]
\end{proof}

\begin{table}[h!]
    \centering \footnotesize
    \begin{tabular}{|c|c|c||c|c|c||c|c|c||c|c|c|}
        \hline
        $x$ & $g_1(x)$ & $g_2(x)$ & $x$ & $g_1(x)$ & $g_2(x)$ & $x$ & $g_1(x)$ & $g_2(x)$ & $x$ & $g_1(x)$ & $g_2(x)$ \\
        \hline \hline
        0.7598 & 0.4210 & 0.4189 & 0.7797 & 0.4139 & 0.4111 & 0.8472 & 0.3678 & 0.3622 & 0.9350 & 0.2338 & 0.2249 \\ \hline
        0.7600 & 0.4209 & 0.4188 & 0.7814 & 0.4131 & 0.4103 & 0.8507 & 0.3643 & 0.3586 & 0.9380 & 0.2270 & 0.2180 \\ \hline
        0.7603 & 0.4208 & 0.4187 & 0.7832 & 0.4124 & 0.4095 & 0.8543 & 0.3606 & 0.3547 & 0.9409 & 0.2202 & 0.2112 \\ \hline
        0.7606 & 0.4207 & 0.4186 & 0.7851 & 0.4115 & 0.4085 & 0.8579 & 0.3567 & 0.3507 & 0.9437 & 0.2134 & 0.2045 \\ \hline
        0.7609 & 0.4206 & 0.4185 & 0.7871 & 0.4106 & 0.4075 & 0.8615 & 0.3528 & 0.3465 & 0.9465 & 0.2065 & 0.1975 \\ \hline
        0.7613 & 0.4205 & 0.4184 & 0.7892 & 0.4095 & 0.4064 & 0.8651 & 0.3486 & 0.3422 & 0.9492 & 0.1996 & 0.1907 \\ \hline
        0.7617 & 0.4204 & 0.4183 & 0.7913 & 0.4085 & 0.4053 & 0.8688 & 0.3442 & 0.3377 & 0.9518 & 0.1927 & 0.1839 \\ \hline
        0.7621 & 0.4203 & 0.4181 & 0.7935 & 0.4074 & 0.4041 & 0.8725 & 0.3397 & 0.3330 & 0.9543 & 0.1860 & 0.1772 \\ \hline
        0.7626 & 0.4201 & 0.4180 & 0.7958 & 0.4062 & 0.4028 & 0.8762 & 0.3350 & 0.3281 & 0.9567 & 0.1793 & 0.1706 \\ \hline
        0.7631 & 0.4200 & 0.4178 & 0.7982 & 0.4048 & 0.4014 & 0.8799 & 0.3301 & 0.3230 & 0.9590 & 0.1728 & 0.1641 \\ \hline
        0.7637 & 0.4198 & 0.4176 & 0.8007 & 0.4034 & 0.3999 & 0.8836 & 0.3251 & 0.3178 & 0.9612 & 0.1663 & 0.1577 \\ \hline
        0.7643 & 0.4196 & 0.4174 & 0.8033 & 0.4019 & 0.3983 & 0.8873 & 0.3198 & 0.3124 & 0.9633 & 0.1600 & 0.1515 \\ \hline
        0.7650 & 0.4194 & 0.4171 & 0.8060 & 0.4003 & 0.3965 & 0.8909 & 0.3146 & 0.3070 & 0.9654 & 0.1535 & 0.1452 \\ \hline
        0.7657 & 0.4191 & 0.4169 & 0.8088 & 0.3985 & 0.3947 & 0.8945 & 0.3092 & 0.3014 & 0.9674 & 0.1472 & 0.1390 \\ \hline
        0.7665 & 0.4189 & 0.4166 & 0.8116 & 0.3967 & 0.3927 & 0.8981 & 0.3035 & 0.2957 & 0.9693 & 0.1411 & 0.1330 \\ \hline
        0.7673 & 0.4186 & 0.4163 & 0.8145 & 0.3948 & 0.3907 & 0.9017 & 0.2977 & 0.2897 & 0.9711 & 0.1351 & 0.1271 \\ \hline
        0.7682 & 0.4183 & 0.4159 & 0.8175 & 0.3927 & 0.3884 & 0.9052 & 0.2919 & 0.2838 & 0.9728 & 0.1293 & 0.1215 \\ \hline
        0.7692 & 0.4179 & 0.4156 & 0.8206 & 0.3904 & 0.3861 & 0.9087 & 0.2859 & 0.2776 & 0.9744 & 0.1237 & 0.1160 \\ \hline
        0.7702 & 0.4176 & 0.4152 & 0.8237 & 0.3881 & 0.3836 & 0.9122 & 0.2797 & 0.2713 & 0.9759 & 0.1183 & 0.1109 \\ \hline
        0.7713 & 0.4172 & 0.4147 & 0.8269 & 0.3857 & 0.3810 & 0.9156 & 0.2734 & 0.2650 & 0.9773 & 0.1132 & 0.1059 \\ \hline
        0.7725 & 0.4167 & 0.4142 & 0.8301 & 0.3831 & 0.3783 & 0.9190 & 0.2670 & 0.2584 & 0.9787 & 0.1080 & 0.1009 \\ \hline
        0.7738 & 0.4163 & 0.4137 & 0.8334 & 0.3803 & 0.3754 & 0.9223 & 0.2606 & 0.2519 & 0.9800 & 0.1030 & 0.0961 \\ \hline
        0.7752 & 0.4157 & 0.4131 & 0.8368 & 0.3774 & 0.3723 & 0.9256 & 0.2539 & 0.2452 &        &        &        \\ \hline
        0.7766 & 0.4152 & 0.4125 & 0.8402 & 0.3744 & 0.3691 & 0.9288 & 0.2473 & 0.2385 &        &        &        \\ \hline
        0.7781 & 0.4145 & 0.4119 & 0.8437 & 0.3711 & 0.3657 & 0.9319 & 0.2406 & 0.2318 &        &        &        \\ \hline
    \end{tabular}
    \caption{Evaluations of $g_1$ and $g_2$ to precision $10^{-4}$. We require that for consecutive inputs $x_i<x_{i+1}$ in the table, $g_1(x_{i+1})- g_2(x_i)\ge 0$. The values shown in fact satisfy $g_1(x_{i+1})- g_2(x_i)\geq\frac{2}{1000}$ modulo rounding.}
    \label{tab:Gbash}
\end{table}

\subsection{Verification on $I_3$}

\begin{proposition}
\label{prop:I3}
    Claim~\ref{claim:ineq} holds for $x \in I_3$.
\end{proposition}
\begin{proof}
    Taylor expansion of $\log(1-\eps)$ gives that for all $\eps \in (0,1)$, 
    \[
        \eps \lt(\log \fr{1}{\eps} + 1 - \eps\rt)
        \le 
        H(\eps) 
        \le 
        \eps \lt(\log \fr{1}{\eps} + 1\rt)\,.
    \]
    Let $x = 1-\eps$ for $\eps \in [0,0.02]$. Then
    \begin{align*}
        g_1(x) &= \varphi H(2\eps - \eps^2) 
        \ge \varphi \eps (2-\eps) \big(\log \fr{1}{\eps} - \log(2-\eps) + (1-\eps)^2\big)\,, \\
        g_2(x) &= (1-\eps) H(\eps) \le \eps (1-\eps) \big(\log \fr{1}{\eps} + 1\big)\,.
    \end{align*}
    Dividing by $\eps$, it suffices to prove
    \[
        \big((2\varphi-1) + (1-\varphi) \eps\big) \log \fr{1}{\eps}
        \ge 
        (1-\eps) \lt(1 - \varphi(1-\eps)(2-\eps)\rt) + \varphi(2-\eps)\log(2-\eps).
    \]
    Noting $\varphi(1-\eps)(2-\eps)\ge 1$ in the first line below, we next find
    \begin{align*}
        &(1-\eps) \lt(1 - \varphi(1-\eps)(2-\eps)\rt) + \varphi(2-\eps)\log(2-\eps)
        \le 
        2\varphi \log 2
        = 
        (\sqrt5-1)\log 2,
    \\
    &\lt((2\varphi-1) + (1-\varphi) \eps\rt) \log \fr{1}{\eps}
        \ge 
        (2\varphi-1) \log \fr{1}{\eps} 
        \ge 
        (\sqrt5-2) \log 50.
    \end{align*}
    Finally $(\sqrt5-2) \log 50 \ge (\sqrt5-1)\log 2$ because
    \begin{align*}
    \log_2(50)
    &\geq
    \log_2(2^5\cdot 1.5)
    \geq
    5.5
    \\
    &\geq
    3+\sqrt{5}
    % =\frac{\sqrt 5-1}{\sqrt 5-2}.
    =(\sqrt 5-1)/(\sqrt 5-2).
    \end{align*}
    Hence the proof is complete. 
    Equality holds if and only if $\eps=0$, i.e. $x=1$.
\end{proof}

\begin{proof}[Proof of Claim~\ref{claim:ineq}]
    Follows by combining Claims~\ref{claim:I1}, \ref{claim:I2} and Proposition~\ref{prop:I3}.
\end{proof}

\end{document}